\documentclass[reqno,12pt]{amsart}
\usepackage{latexsym}
\usepackage{amsmath,amsfonts,amssymb,epsfig}
\usepackage{amsthm, amscd}
\usepackage{hyperref}
\usepackage{xcolor}
\usepackage{setspace}
\usepackage{overpic}
\usepackage{xypic}
    
\usepackage{tikz}
\usetikzlibrary{arrows,chains,matrix,positioning,scopes}

\makeatletter
\tikzset{join/.code=\tikzset{after node path={%
\ifx\tikzchainprevious\pgfutil@empty\else(\tikzchainprevious)%
edge[every join]#1(\tikzchaincurrent)\fi}}}
\makeatother

\tikzset{>=stealth',every on chain/.append style={join},
         every join/.style={->}}

\newtheorem{thm}{Theorem}[section]
\newtheorem*{mthm}{Main Theorem}
\newtheorem{cor}[thm]{Corollary}
\newtheorem{lem}[thm]{Lemma}
\newtheorem{conj}[thm]{Conjecture}
\newtheorem{prop}[thm]{Proposition}

\theoremstyle{definition}
\newtheorem{rem}[thm]{Remark}
\newtheorem{prob}[thm]{Question}

\newtheorem{defin}[thm]{Definition}

\theoremstyle{definition}

\newcommand{\no}{\noindent}

\newcommand{\bex}{\begin{example}\em}
	\newcommand{\eex}{\end{example}}

\newcommand{\link}{\operatorname{link}}
\newcommand{\Link}{\operatorname{Link}}
\def\co{\colon\!}
\newcommand{\map}{\operatorname{map}}
\newcommand{\BLink}{\operatorname{BrLink}}
\newcommand{\slink}{\operatorname{slink}}
\newcommand{\holim}{\operatorname{holim}}

\def\R{\mathbb{R}}

\def\T{\mathbb{T}}
\def\C{\mathcal{C}}
\def\P{\mathcal{P}}
\def\LL{\mathcal{L}}
\def\x{\times}

\def\Cnf{\text{\rm Conf}}
 
\def\map{\mathrm{map}}
\def\id{\mathrm{id}}

\graphicspath{{./pict/},{./../pict/},{./../../pict/}}

\numberwithin{equation}{section} 
\numberwithin{thm}{section} 

\title{Homotopy string links and the $\kappa$-invariant}

\author{F. R. Cohen}

\address{
University of Rochester,
Rochester, New York 14627 } 
\email{cohf@math.rochester.edu}

\author{R. Komendarczyk}\thanks{The second author acknowledges the support of DARPA grant YFA N66001-11-1-4132 and NSF grant DMS 1043009.}

\address{
Tulane University,
New Orleans, Louisiana 70118 } 
\email{rako@tulane.edu}

\author{R. Koytcheff}\thanks{The third author was supported by a PIMS postdoctoral fellowship and by NSERC}

\address{
University of Massachusetts,
Amherst, Massachusetts 01003
} 
\email{koytcheff@math.umass.edu}

\author{C. Shonkwiler}

\address{
Colorado State University,
Fort Collins, Colorado 80521 } 
\email{clayton@math.colostate.edu}

\begin{document}

\begin{abstract}
	Koschorke introduced a map from the space of closed $n$-component links to the space of maps from the $n$-torus to the ordered configuration space of $n$-tuples of points in $\R^3$.  He conjectured that this map separates homotopy links. The purpose of this paper is to construct an analogous map for string links, and to prove (1) this map in fact separates homotopy string links, and (2) Koschorke's original map factors through the map constructed here together with an analogue of Markov's closure map defined on the level of certain function spaces.
\end{abstract}

\maketitle

\vspace{-1cm} 

\section{Introduction}

Recall from~\cite{Koschorke:1991a,Koschorke:1997} that Koschorke's $\kappa$-invariant assigns to each link-homotopy class of $n$--component links an associated homotopy class of maps from a torus to a configuration space. The precise definition will be given below, but the basic appeal of Koschorke's construction is threefold. 

First, one of Koschorke's primary motivations was to define an invariant which extends Milnor's $\overline{\mu}$-invariants~\cite{Milnor:1954,Milnor:1957} to higher-dimensional links, and indeed he and others (e.g.,~\cite{Koschorke:2004,Munson:2011}) have given very precise descriptions of the higher-dimensional situation.

Second, despite its rather abstract appearance, the $\kappa$-invariant gives a natural way of defining numerical link-homotopy invariants.
Since link-homotopy invariants are otherwise relatively hard to come by, this is an important potential source for numerical link-homotopy invariants which may be applied to, for example, problems in plasma physics (cf.~\cite{Komendarczyk:2009,Komendarczyk:2010,DeTurck-Gluck-Komendarczyk-Melvin:2013,DeTurck-Gluck-Komendarczyk-Melvin:2013a}).
Finally, it remains an open question first posed by Koschorke whether the $\kappa$-invariant is a complete invariant of homotopy links (i.e.~links up to link-homotopy).  It is known to separate Borromean links~\cite{Koschorke:1997,Cohen-Komendarczyk-Shonkwiler:2015} as well as all 2- and 3-component links~\cite{DeTurck-Gluck-Komendarczyk-Melvin:2013}, but the situation remains unclear for links with 4 or more components.

If the $\kappa$-invariant really does separate all homotopy links, then it provides an alternative classification of homotopy links to that of Habegger and Lin~\cite{Habegger-Lin:1990}.  Their classification was based on a Markov-type theorem for closures of homotopy string links, which suggests a strategy for answering Koschorke's question: first, define an analog $\check{\kappa}$ of $\kappa$ for homotopy string links and show that it separates them, and then show that the ``closure'' map $\check{\kappa} \mapsto \kappa$ is compatible with Habegger and Lin's classification.

In this paper we carry out the first part of this plan, showing that $\check{\kappa}$ separates homotopy string links. Although it may not be explicitly apparent in what follows, we drew significant inspiration from the close connection between the $\kappa$-invariant and Milnor's $\overline{\mu}$-invariants, based on Habegger and Lin's observation that Milnor's invariants are
integer invariants that separate string links. 

The paper is structured as follows. In Section~\ref{sec:kappa}, we recall the definition of the $\kappa$-invariant, define $\check{\kappa}$, and give the precise statement of the main theorem. Section~\ref{sec:kappa-kappa} is devoted to the translation between $\kappa$ and $\check{\kappa}$. The key advantage of working with homotopy string links is that they form a group.  In Section~\ref{S:string-links}, we describe monoid structures on the space of string links and our mapping space via an action of the little intervals operad.  We furthermore show that $\check{\kappa}$ is a monoid homomorphism. We put all these pieces together to prove the main theorem in Section~\ref{sec:proof of main theorem}. 
In Section \ref{sec:Taylor-tower} we describe the $\kappa$-invariant in terms of the functor calculus of Goodwillie and Weiss, in particular the Taylor tower for the space of link maps.

\section{The \texorpdfstring{$\kappa$}{k}-invariant}\label{sec:kappa}
Fix an $n$-tuple $\mathbf{a}=(a_1,\ldots,a_n)$ of distinct points in $\R^3$ and  consider the function space ${\link_{\mathbf{a}}(\sqcup^{n}_{i=1} S^1;\R^3) }$ of pointed   $n$-component smooth link maps ($n\geq 2$), i.e. smooth immersions, 
\begin{equation}\label{eq:link-map}
\begin{split}
L = (L_1,...,L_n) \co & (S^1,s_1)\sqcup\ldots \sqcup (S^1,s_n)  \xrightarrow{\qquad} \R^3,\\
& L_i(s_i)=a_i,\quad L_i(S^1)\cap L_j(S^1)= \emptyset,\ i\neq j,
\end{split}
\end{equation}
Two link maps $L$ and $L'$ are {\em link homotopic}  if 
there exists a  homotopy ${H \co \left(\bigsqcup^n_{i=1} S^1\right)\times I\to \R^3}$ connecting $L$ and $L'$ through pointed link maps. We denote the set of equivalence classes of pointed $n$-component link maps by $\Link(n)$, i.e. 
\begin{equation}\label{eq:LM(n)}
\Link(n)=\pi_0(\link_{\mathbf{a}}(\sqcup^{n}_{i=1} S^1; \R^3)).
\end{equation}

Recall, that for any $X$,
\[
\Cnf(X,n) := \{(x_1, \ldots , x_n) \in X^n\ |\ x_i \neq x_j \text{ for } i \neq j\},
\]
and that an inclusion of spaces $f\co X \hookrightarrow Y$ induces a map $\Cnf(f,n)\co \Cnf(X,n) \hookrightarrow \Cnf(Y,n)$.  
Notice that $\Cnf(\sqcup_{i=1}^n S^1, n)$ contains a copy of the $n$-torus $\T^n$, namely the connected component where the $i$th configuration point lies on the $i$th copy of $S^1$.
Now abbreviate 
\[
{\Cnf(n):=\Cnf(\R^3,n)}.
\]
The $\kappa$-invariant~\cite{Koschorke:1991a,Koschorke:1997} for classical links (i.e. 1-dimensional closed links in $\R^3$) is then defined via the map  
\begin{equation*}
\begin{split}
\link_{\mathbf{a}}(\sqcup^{n}_{i=1} S^1; \R^3) & \to \map_\mathbf{a}(\T^n, \Cnf(n)), \\
L &\mapsto \Cnf(L,n)|_{\T^n},
\end{split}
\end{equation*}  
where $\map_\mathbf{a}(\T^n, \Cnf(n))$ is the function space of continuous pointed maps from 
$\T^n$ to $\Cnf(\R^3, n)$.
Define $\kappa$ as the induced map on $\pi_0$:
\begin{equation}\label{eq:kappa}
\begin{split}
\kappa \co \Link(n) & \to [\T^n, \Cnf(n)],\\
\kappa([L]) & = [\Cnf(L,n)|_{\mathbb{T}^n}]
\end{split}
\end{equation}
where  $[\T^n, \Cnf(n)]=\pi_0(\map_\mathbf{a}(\T^n, \Cnf(n))$.

Note that we use pointed maps in \eqref{eq:link-map}--\eqref{eq:kappa} mainly for convenience (see 
\ref{app:basepoints}).  In his works~\cite{Koschorke:1991a,Koschorke:1997,Koschorke:2004}
Koschorke  introduced the following central question concerning the $\kappa$-invariant.
\begin{prob}[Koschorke~\cite{Koschorke:1997}]\label{q:koschorke}
	Is the $\kappa$-invariant injective and therefore a complete invariant of $n$-component classical links up to link homotopy?
\end{prob}

\no In \cite{Koschorke:1997} Koschorke showed that $\kappa$  is injective on link-homotopy classes of so-called Borromean link maps $\BLink(n)$\footnote{This space is denoted by $BLM(n)$ in \cite{Koschorke:1997}.} \cite[Theorem 6.1 and Corollary 6.2]{Koschorke:1997}. More recently, a positive answer to the above question was given for $n=3$~\cite{DeTurck-Gluck-Komendarczyk-Melvin:2013}.  A modern treatment of the subject via the Goodwillie calculus is presented in \cite{Munson:2011}.

In this paper, we study an analog of the map in \eqref{eq:kappa} in the case $\Link(n)$ is replaced by the group of homotopy string links $\mathcal{H}(n)$, $n\geq 2$. Recall that  homotopy string links were introduced by Habegger and Lin~\cite{Habegger-Lin:1990} to address the classification problem for closed homotopy links.   
\begin{figure}[htbp]
	\begin{overpic}[width=.3\textwidth]{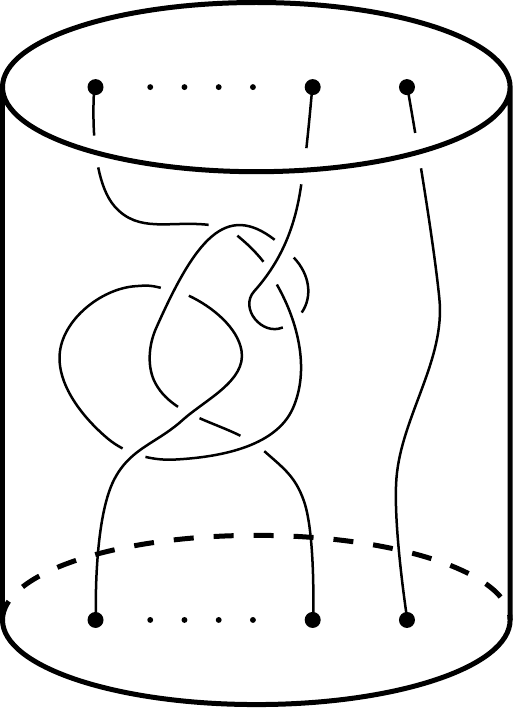}
		\put(12,7){$a_1$}
		\put(43,7){$a_{n-1}$}
		\put(57,7){$a_n$}
	\end{overpic}
	\caption{{\small $n$-component string link.} }
	\label{fig:s-link}
\end{figure} 
Following Habegger and Lin, consider a cylinder $\mathcal{C}=D^2 \times I$ in $\R^3$, where $D^2$ is the unit disk and $I=[0,1]$. Let $\mathbf{a}=(a_1,\ldots, a_n)$ be an $n$-tuple of marked points in $D^2$. A {\em string link} $\sigma$ is a smooth immersion of $n$ disjoint copies of the unit interval $I=I_i\cong [0,1]$ into  $\mathcal{C}$, i.e.
\[
\sigma = (\sigma_1,...,\sigma_n) \co I_1\sqcup\ldots \sqcup I_n \longrightarrow \mathcal{C},\qquad \sigma_i(I_i)\cap \sigma_j(I_j)=\emptyset,\ i\neq j, 
\]
where the $i$th  component satisfies
\begin{equation}\label{eq:sigma-endpoints}
\sigma_i(0)=\sigma|_{I_i}(0)=(a_i,0) \qquad \text{and}\qquad  \sigma_i(1)=\sigma|_{I_i}(1)=(a_i, 1).
\end{equation}
Here $\sigma_i \co I\longrightarrow \mathcal{C}$ denotes the $i$th strand of $\sigma$. Further, for technical reasons we assume that each strand $\sigma_i$ satisfies the following condition
\begin{equation}\label{eq:sigma-support}
\begin{minipage}{.6\textwidth}
\begin{center}
$\sigma_i\bigl((0,1)\bigr)$ is contained in the interior of $\mathcal{C}$ and each strand meets $D^2\times\{0\}$ and $D^2 \times\{1\}$ orthogonally.
\end{center}
\end{minipage}
\end{equation}
Denote the resulting function space by $\slink_{\mathbf{a}}(\sqcup^n_{i=1} I_i,\mathcal{C})$. Then 
\begin{equation*}
\mathcal{H}(n)=\pi_0(\slink_{\mathbf{a}}(\sqcup^n_{i=1} I_i,\mathcal{C}))
\end{equation*}
is the set (in fact, group) of link-homotopy classes of string links, a.k.a.~homotopy string links.
\no For related constructions of the function spaces of string links see the recent work~\cite{Koytcheff-Munson-Volic:2013}.
Notice that $\mathcal{H}(n)$ is the analog of $\Link(n)$: the former denotes link-homotopy classes of string links, while the latter denotes link-homotopy classes of closed links.

We will consider an appropriate subspace $\map_{\mathbf{a}}(I^n,\Cnf(\mathcal{C},n))$ of maps from the $n$-cube $I^n$ to the configuration space of $n$ points in $\mathcal{C}$; this subspace will contain the image of the map $\Cnf(\sigma)$ induced by a string link $\sigma$ on configuration spaces. 
Define $\map_{\mathbf{a}}(I^n,\Cnf(\mathcal{C},n))$ to be the subspace of
$f=(f_1,\ldots,f_n)$
which satisfy the following conditions:
\begin{itemize}
	\item[] (endpoints) for all $i=1,\dots ,n$,
	\begin{equation}\label{eq:endpoints}
	f_i|_{I_1\times \ldots \times I_{i-1}\times \{0\}\times I_{i+1}\times \ldots\times I_n}= c_{(a_i,0)},\quad f_i|_{I_1\times \ldots \times I_{i-1}\times \{1\}\times I_{i+1}\times \ldots\times I_n}= c_{(a_i,1)},
	\end{equation}
	\qquad \qquad \quad where $c_x$ denotes the constant map at a point $x$.
	%
	%
	\item[] (support) the image of the interior of $I^n$ under $f$ is in the interior of $\Cnf(\mathcal{C},n)$
	\begin{equation}\label{eq:support}
	f(\operatorname{int}(I^n))\subset \operatorname{int}(\Cnf(\mathcal{C},n)).
	\end{equation}
\end{itemize} 
Due to condition \eqref{eq:endpoints} we say that maps in $\map_\mathbf{a}(I^n,\Cnf(\mathcal{C},n))$ are {\em based at $\mathbf{a}=(a_1,\ldots, a_n)$}. 

Next, we can define the analog of the $\kappa$-map in \eqref{eq:kappa} for homotopy string links.  We start with the map 
\begin{equation*}
\begin{split}
\slink_{\mathbf{a}}(\sqcup^{n}_{i=1} I_i; \C) & \longrightarrow \map_\mathbf{a}(I^n, \Cnf(\C,n)), \\
\sigma &\longmapsto \Cnf(\sigma, n)|_{I^n}
\end{split}
\end{equation*}  
and define $\check{\kappa}$ to be the induced map on $\pi_0$,
\begin{equation}\label{eq:kappa-hat}
\begin{split}
\check{\kappa}\co \mathcal{H}(n) & \longrightarrow \mathcal{M}(n),\\
\check{\kappa}([\sigma]) & = [\Cnf(\sigma, n)|_{I^n}]
\end{split}
\end{equation}
where $\mathcal{M}(n)=\pi_0(\map_{\mathbf{a}}(I^n, \Cnf(\mathcal{C},n))$. It is straightforward to show that $\check{\kappa}$ is well defined, thanks to properties \eqref{eq:sigma-endpoints} and \eqref{eq:sigma-support}. The main theorem of this paper  can now be stated:
\begin{mthm}
	The map $\check{\kappa}: \mathcal{H}(n) \to \mathcal{M}(n)$ separates homotopy string links.
\end{mthm}
\no Equivalently, the main theorem says that $\check{\kappa}$ is injective.  Our mapping space is equivalent to a stage of the Taylor tower for the space of link maps \cite{Weiss:1996}.  However, this fact by itself does not facilitate the proof of our main theorem, so we postpone its discussion until Section \ref{sec:Taylor-tower}.

\section{Transition between  \texorpdfstring{$\kappa$}{k} and  \texorpdfstring{$\check{\kappa}$}{kk}}\label{sec:kappa-kappa}
There is a relation between the $\kappa$-invariant of Koschorke defined in \eqref{eq:kappa} and 
the $\check{\kappa}$ of \eqref{eq:kappa-hat} for string links,
 via the Markov ``closure'' operation: $\widehat{\,\cdot\,}:\mathcal{H}(n)\longrightarrow \Link(n)$. In the work of Habegger and Lin~\cite{Habegger-Lin:1990}, the closure $\widehat{\sigma}$ of a given string link $\sigma\in \mathcal{H}(n)$ is obtained by adding ``trival'' closing strands outside of the embedded cylinder $\mathcal{C}$. For our purposes we use an equivalent but different method of closing a string link. Consider an immersion 
\begin{equation*}
b: \mathcal{C}\longmapsto \mathcal{T}\subset \R^3,
\end{equation*}
which ``bends'' the cylinder $\mathcal{C}$ into a solid torus $\mathcal{T}$ embedded in $\R^3$. Under $b$ the bottom disk $D_0=D^2\times \{0\}$ of $\mathcal{C}$ is identified with the top disk $D_1=D^2\times\{1\}$ into a cross-sectional disk in $\mathcal{T}$. The endpoints of strands are glued under $b$, i.e. $(a_i,0)\sim (a_i,1)$, and the resulting $n$-tuple of points is again denoted by $\mathbf{a}=(a_1,\ldots, a_n)$. Clearly, for any string link $\sigma$ in $\mathcal{C}$, $b(\sigma)$ is a closed link, with each $i$-component based at $a_i$ and contained in $\mathcal{T}\subset\R^3$, therefore we obtain a map
\[ 
b\co \slink_{\mathbf{a}}(\sqcup^n_{i=1} I_i,\mathcal{C})\longrightarrow \link_{\mathbf{a}}(\sqcup^n_{i=1} S^1_i,\R^3)
\]
equivalent (modulo link homotopy) to the usual Markov closure map $\widehat{\cdot}$.
Since $b$ is an embedding on $\operatorname{int}(\mathcal{C})$ it descends to a map of configuration spaces 
\[
\hat{b}\co \Cnf(\operatorname{int}(\mathcal{C}),n)\longrightarrow \Cnf(\mathcal{T},n)\subset\Cnf(\R^3,n).
\]

To get an induced map on function spaces, we need to impose one more condition on our maps, so as to make them ``roundable.''  Define $\map_{\mathbf{a}}^\circ(I^n,\Cnf(\mathcal{C},n))$ to be the subspace of maps in $\map_{\mathbf{a}}(I^n,\Cnf(\mathcal{C},n))$
satisfying the following condition:
	\begin{itemize}
	\item[] (periodicity) for all $i=1,\dots ,n$ and $k\neq i$,
	\begin{equation}\label{eq:periodicity}
	 f_k|_{I_1\times \ldots \times I_{i-1}\times \{0\}\times I_{i+1}\times \ldots\times I_n}=f_k|_{I_1\times \ldots \times I_{i-1}\times \{1\}\times I_{i+1}\times \ldots\times I_n}.
	\end{equation}
	\end{itemize}
The image of $\slink_{\mathbf{a}}(\sqcup^n_{i=1} I_i,\mathcal{C})$ under $\check{\kappa}$ is clearly contained in this subspace $\map_{\mathbf{a}}^\circ(I^n,\Cnf(\mathcal{C},n))$.  By abuse of notation, we will also use $\check{\kappa}$ to denote the map with this smaller codomain.
	
Now 
$\hat{b}$ gives a well defined map between the following function spaces\footnote{We use the same symbols $b$ and $\hat{b}$ for these various maps, as it should be obvious from the context which one is applied.}
\[
\hat{b}\co \map_{\mathbf{a}}(I^n, \Cnf(\mathcal{C},n)) \longrightarrow \map_{\mathbf{a}}^\circ(\T^n, \Cnf(n)),
\]
defined by $\hat{b}(f)=b\circ f$. Note that $\hat{b}(f)$ is a well defined map on $\T^n$ thanks to~\eqref{eq:endpoints}, \eqref{eq:support}, and~\eqref{eq:periodicity}. 
The definitions of $b$, $\hat{b}$, $\kappa$, and $\check{\kappa}$ immediately imply that $\hat{b}\circ \check{\kappa}=\kappa\circ b$ on function spaces.
So if we define $\mathcal{M}^\circ(n)= \pi_0(\map_{\mathbf{a}}^\circ(I^n,\Cnf(\mathcal{C},n)))$, 
 we obtain the following diagram at the level of path-components:
\begin{equation}\label{diag:kappa-to-kappa}
\begin{tikzpicture}[baseline=(current bounding box.center),description/.style={fill=white,inner sep=2pt}]
\matrix (m) 
[
matrix of math nodes
, row sep=3em
, column sep=2.5em, 
text height=1.0ex, text depth=0.25ex
]
{ \mathcal{H}(n)  & & \mathcal{M}^\circ (n)\\
	\Link(n) & & \text{$[\T^n, \Cnf(\R^3,n)]$}\\ 
};
\path[->,font=\scriptsize]

(m-1-1) edge node[auto] {$\check{\kappa}$} (m-1-3)
(m-2-1) edge node[auto] {$\kappa$} (m-2-3)

(m-1-1) edge node[auto] {$b$} (m-2-1)

(m-1-3) edge node[auto] {$\hat{b}$} (m-2-3);
\end{tikzpicture}
\end{equation}
The map $\mathcal{H}(n) \to \mathcal{M}(n)$ from our main theorem obviously factors through the top horizontal map above.  Hence our main theorem implies injectivity of the latter map.

\section{Relevant algebraic structures}\label{S:string-links}

\subsection{Group structure on \texorpdfstring{$\mathcal{H}(n)$}{H(n)}} 
Following \cite{Habegger-Lin:1990}, $\mathcal{H}(n)$ has a group structure where the multiplication is defined by ``stacking'' string links in a vertical fashion. 
First define two maps 
${l,u \co I \to I}$,
lower and upper, by rescaling the 
interval
to its lower and upper half respectively.  Define $l^{-1}$ and $u^{-1}$ as (the restrictions to $I$ of) the inverses of (the affine-linear extensions to $\R$ of) $l$ and $u$.
Specifically, 
\begin{equation}\label{eq:u-and-l}
\begin{split}
l, u \co 
I  \longrightarrow 
I, &\qquad l^{-1},u^{-1}\co I  \longrightarrow I \\
l(t)  ={\textstyle \frac{1}{2} t}, &\qquad l^{-1}(t) = 2t,\\
u(t)  ={\textstyle \frac{1}{2} t+\frac{1}{2}}, & \qquad u^{-1}(t) = 2t-1.
\end{split}
\end{equation}
For $\sigma, \sigma':I_1\sqcup \ldots \sqcup I_n\longmapsto \mathcal{C}$, the product $\sigma\ast \sigma'$ is given on each factor by 
\begin{equation}\label{eq:s-link-prod}
(\sigma\ast \sigma')_i(t_i)=\begin{cases}
(l \x \id_{D^2}) \circ \sigma_i \circ l^{-1} (t_i), & \quad t_i\in [0, \frac 12],\\
(u \x \id_{D^2}) \circ \sigma'_i \circ u^{-1} (t_i), & \quad t_i\in [\frac 12, 1],
\end{cases}
\end{equation}
i.e. the usual product of paths $(l \x \id_{D^2}) \circ \sigma$ and $(u \x \id_{D^2}) \circ \sigma'$ in $\mathcal{C}$. The result satisfies~\eqref{eq:sigma-endpoints} and~\eqref{eq:sigma-support}, so it is a well defined element of $\slink_{\mathbf{a}}(\sqcup^n_{i=1} I_i,\mathcal{C})$.

In fact, Habegger and Lin showed that $\mathcal{H}(n)$ is a group by exhibiting it as an extension of groups in the split short exact sequence \cite[Lemma 1.8]{Habegger-Lin:1990}
\begin{equation}\label{eq:H(n)-short-exact}
1 \longrightarrow \mathcal{K}_i(n-1) \xrightarrow{\quad\quad} \mathcal{H}(n) \xrightarrow{\quad\delta_i\quad}\mathcal{H}_i(n-1)\longrightarrow 1.
\end{equation}
Here $\mathcal{H}_i(n-1)$ is the copy of $\mathcal{H}(n-1)$ obtained as the image of the map $\delta_i$ which deletes the $i$th strand; i.e., $\mathcal{H}_i(n-1)$ consists of all $(n-1)$-component string links based at ${\mathbf{a}_{\widehat{i}}=(a_1,\ldots, \widehat{a_i},\ldots, a_n)}$. The normal subgroup $\mathcal{K}_i(n-1)$ is isomorphic to the {\em reduced free group} $RF(n-1)$ on $n-1$ generators. 

\begin{rem}
This exact sequence above is analogous to the sequence 
\begin{equation}\label{eq:PB(n)-short-exact}
1 \longrightarrow F(n-1) \xrightarrow{\quad\quad} PB(n) \xrightarrow{\quad\delta_i\quad} PB_i(n-1)\longrightarrow 1.
\end{equation}
where $PB(n)$ is the pure braid group on $n$ strands, $PB_i(n)$ is the copy of $PB(n-1)$ obtained by deleting the $i$th strand of an element of $PB(n)$, and $F(n-1)$ is the free group on $n-1$ generators.  
In fact, the sequence (\ref{eq:H(n)-short-exact}) can be obtained from (\ref{eq:PB(n)-short-exact}) by the quotient $F(n) \to RF(n)$ \cite{Goldsmith:1973}, \cite[p. 399]{Habegger-Lin:1990}.
Thus the inclusion of the space of pure braids into the space of link maps of string links induces a surjection on path-components $PB(n) \to \mathcal{H}(n)$, and any string link is link-homotopic to a pure braid.  
\end{rem}

Building on Habegger and Lin's construction, consider the homomorphism
\begin{equation}\label{eq:delta}
\delta = \prod_{i=1}^n \delta_i\co \mathcal{H}(n) \xrightarrow{\qquad \quad} \prod^n_{i=1} \mathcal{H}_i(n-1),
\end{equation}
where $\delta_i$ is as defined in \eqref{eq:H(n)-short-exact}. Then 
\begin{equation*}
\ker \delta = \mathcal{K}_1(n-1) \cap \mathcal{K}_2(n-1) \cap \ldots \cap \mathcal{K}_n(n-1).
\end{equation*}
The string links representing elements of $\ker \delta$ have a natural geometric meaning:  they are precisely the string links which become link-homotopically trivial after removing \emph{any} of their components, which we call\footnote{Some authors use the term {\em Brunnian links}.} {\em Borromean string links} and denote by $Br\mathcal{H}(n)$; i.e., 
\begin{equation}\label{eq:BH(n)-def}
Br\mathcal{H}(n):=\ker\delta.
\end{equation}
\no The structure of $Br\mathcal{H}(n)$ is well understood, in particular $Br\mathcal{H}(n)$  is isomorphic to a direct product of $(n-2)!$ copies of integers, generated by length 
$(n-1)$ iterated commutators in $RF(n-1)$, (see Lemma 2.2 in \cite{Cohen-Komendarczyk-Shonkwiler:2015}). The subspace of $\BLink(n)$ of Borromean links in $\Link(n)$, i.e. links which become trivial after removing any one of the components, is the image of $Br\mathcal{H}(n)$ under the closure map.

We summarize the needed results of \cite{Cohen-Komendarczyk-Shonkwiler:2015}; note that $(ii)$ has also been proved by Koschorke~\cite{Koschorke:1997}. 
\begin{thm}\label{thm:old-main} \ 
	\begin{itemize}
		\item[$(i)$] The closure $b:\mathcal{H}(n)\longrightarrow \Link(n)$ restricted to $Br\mathcal{H}(n)$ is one-to-one; i.e., $b$ induces an isomorphism of $Br\mathcal{H}(n)$ with $\BLink(n)$.
		\item[$(ii)$] $\kappa$ given in \eqref{eq:kappa} is one-to-one on $\BLink(n)$.
	\end{itemize}
\end{thm}
\no As a corollary we obtain:
\begin{lem}\label{lem:check-kappa-on-BH(n)}
	The map $\check{\kappa}$ is one-to-one on $Br\mathcal{H}(n)$. 
\end{lem}
\begin{proof}
	The claim follows from diagram~\eqref{diag:kappa-to-kappa} where both $b$ and $\kappa$ are one-to-one when restricted  to $Br\mathcal{H}(n)$ and $\BLink(n)$ respectively. 
\end{proof}

\subsection{A monoid structure on \texorpdfstring{$\mathcal{M}(n)$}{M(n)}}
\no We now define a multiplication on $\map_\mathbf{a}(I^n, \Cnf(\mathcal{C},n))$ which turns $\mathcal{M}(n)$ into a monoid and makes $\check{\kappa}$ into a monoid homomorphism. 

Subdivide the unit cube 
\[
I^n=I_1\times\ldots\times I_n,\quad I_i=[0,1],
\]
with coordinates $(t_1,\ldots,t_n)$, into $2^n$ subcubes simply by splitting each edge into half. Label the first half $[0,\frac 12]\subset I_i$ by $l$ and the second half $[\frac 12, 1]\subset I_i$ by $u$.  Then any subcube is determined by an $n$-tuple 
$\mathbf{z}=(z_1,z_2,\ldots,z_n)$, where each $z_j$ is either
$l$ or $u$.
This sequence of $l$'s and $u$'s may be viewed as a function $\mathbf{z} = \mathbf{z}(t_1,...,t_n)$.
Furthermore, we write $\mathbf{z}^{-1}$ for $(z_1^{-1},...,z_n^{-1})$, recalling the definitions in \eqref{eq:u-and-l}.  Then $\mathbf{z}^{-1}(t_1,...,t_n)$ can be interpreted as the map that rescales the subcube containing $(t_1,...,t_n)$ to the unit cube.

Given maps $f=(f_1,...,f_n)$ and $g=(g_1,...,g_n)$ in $\map_\mathbf{a}(I^n, \Cnf(\mathcal{C},n))$, define 
\begin{equation}\label{eq:map-prod}
(f \ast g)_i(t_1,...,t_n)=\begin{cases}
\left( (l \x \id_{D^2}) \circ f_i \circ (\mathbf{z}^{-1}(t_1,...,t_n)) (t_1,...,t_n) \right), & \quad t_i\in [0, \frac 12], \\
\left( (u \x \id_{D^2}) \circ g_i \circ (\mathbf{z}^{-1}(t_1,...,t_n)) (t_1,...,t_n) \right), & \quad t_i\in [\frac 12, 1].
\end{cases}
\end{equation}

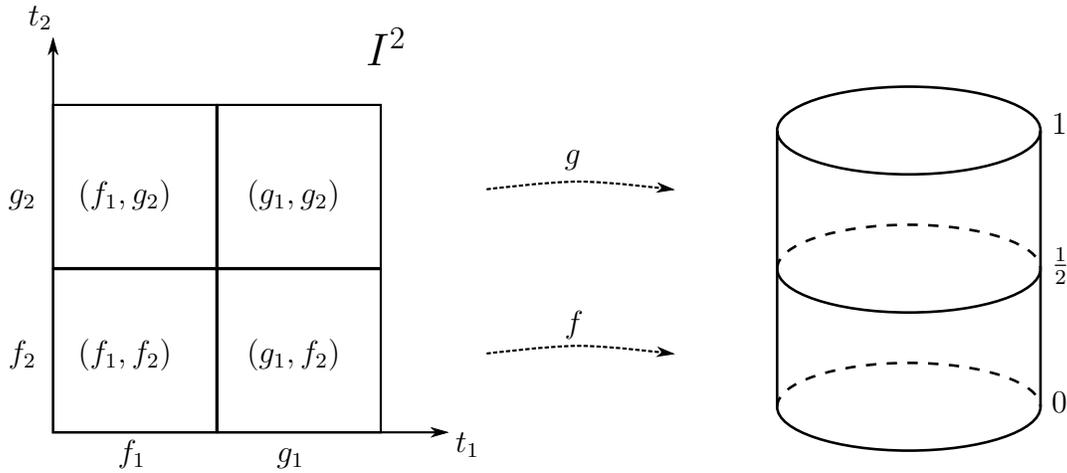
\begin{figure}[htbp]
	\begin{overpic}[width=.8\textwidth]{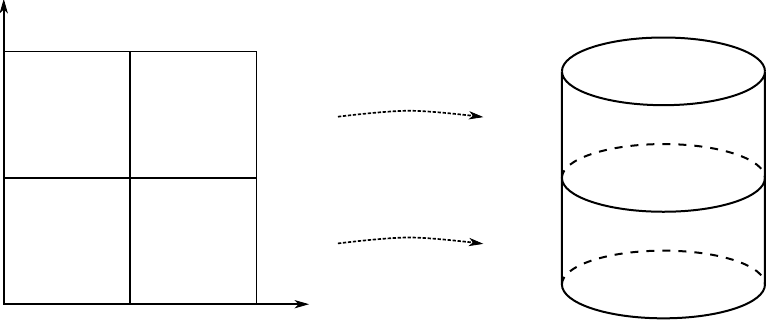}
		\put(-4,9){$f_2$}
		\put(-4,25){$g_2$}
		\put(-2,43){$t_2$}
		\put(7,-1){$f_1$}
		\put(23,-1){$g_1$}
		\put(41,0){$t_1$}
		\put(32,39){\Large{$I^2$}}
		\put(3,9){$(f_1,f_2)$}
		\put(20,9){$(g_1,f_2)$}
		\put(3,25){$(f_1,g_2)$}
		\put(20,25){$(g_1,g_2)$}
		\put(52,12){$f$}
		\put(52,29){$g$}
		\put(101,4){$0$}
		\put(101,18){$\frac{1}{2}$}
		\put(101,32){$1$}
	\end{overpic}
	\caption{{\small The product of two maps $f$ and $g$.} }
	\label{fig:moniod-prod}
\end{figure} 

The fact that this is a well defined, homotopy-associative multiplication on our mapping space can be proven directly using the two defining conditions \eqref{eq:endpoints} and \eqref{eq:support} on our mapping space.  Such a direct proof would be similar to the proof that the fundamental group of a space is a group (i.e., that multiplication on a loop space is homotopy-associative).  Instead of pursuing such an argument, we will equip our mapping space with an action of the {little intervals operad} (a.k.a.~{little 1-cubes operad}) $\mathcal{L}$, a structure that naturally arises on loop spaces \cite{May:1972}, our mapping space, and the space of string links itself.  

\subsection{Little intervals actions}
Recall that the 
\emph{little intervals operad} consists of spaces $\{\LL(k)\}_{k \in \mathbb{N}}$, where $\LL(k)$ is the space of $k$ subintervals $(L_1,...,L_k)$ of $I=[0,1]$ with disjoint interiors.  We may also regard each $L_i$ as the unique affine-linear map $\R \to \R$ which sends $I$ to $L_i$.  From this perspective, we can make sense of the inverse map $L^{-1}$.  Recall that an \emph{action} of the operad $\LL$ on a space $X$ consists of continuous maps 
\begin{equation}
\label{eq:operad-action-maps}
\LL(k) \x X^k \to X
\end{equation}
satisfying the following identity, symmetry, and associativity conditions:
\begin{itemize}
\item
The identity interval $I \in \LL(1)$ induces the identity map $X \to X$.
\item
The map (\ref{eq:operad-action-maps}) is invariant with respect to the (diagonal) action of the symmetric group $\Sigma_k$ on $\LL(k) \x X^k$, given by permuting the $k$ intervals and the $k$ copies of $X$.
\item
The following square commutes:
\begin{equation}\label{eq:associativity}
\begin{tikzpicture}[baseline=(current bounding box.center),description/.style={fill=white,inner sep=2pt}]
\matrix (m) 
[
matrix of math nodes
, row sep=3em
, column sep=2.5em, 
text height=1.0ex, text depth=0.25ex
]
{ 
\LL(m) \x \LL(k_1) \x ... \x \LL(k_m) \x X^{k_1+...+k_m}  & & \LL(m) \x X^m \\
\LL(k_1+...+k_m) \x X^{k_1+...+k_m} & & X\\
};
\path[->,font=\scriptsize]
(m-1-1) edge (m-1-3)
(m-2-1) edge (m-2-3)
(m-1-1) edge (m-2-1)
(m-1-3) edge (m-2-3);
\end{tikzpicture}
\end{equation}
%
%
\end{itemize}
We refer the reader to \cite{May:1972, Markl-Shnider-Stasheff:2002, McClure-Smith:2004} for more information on operads.

Let $\iota_i : \R \to D^2 \x \R$ be the map given by $\iota_i(t) = (a_i, t)$, and let $p: D^2 \x \R \to \R$ be the projection, which satisfies $p \circ \iota_i=\mathrm{id}_\R$ for all $i$.

We begin with an action of $\LL$ on the space of string links.  
To a string link $\sigma$, we extend each component $\sigma_i$ to a map $\overline{\sigma}_i: D^2 \x \R \to D^2 \x \R$, where $i=1,...,n$.  We define this map by
\[
\overline{\sigma}_i(x) :=
\left\{
\begin{array}{ll}
\sigma_i(p(x)) & \mbox{if $p(x) \in I$} \\
x & \mbox{if $p(x) \notin I$}.
\end{array}
\right.
\]
Each $\overline{\sigma}_i$ is not continuous everywhere, but it \emph{is} continuous at any $x$ satisfying the condition 
\begin{equation}
\label{sigma-boundary-are-basepoints}
p(x) \in \partial I \Rightarrow x \in \{a_i\} \x \partial I.
\end{equation}

Now let $(L_1,...,L_k)\in \LL(k)$ be $k$ little intervals, and let $\sigma^1,...,\sigma^k$ be $k$ string links.\footnote{Note that superscripts indicate different string links, while subscripts still denote components of a string link.  Thus each $\sigma^j = (\sigma^j_1,...,\sigma^j_n)$.  When discussing an arbitrary index, will use $i$ for the index for the component and $j$ for the index of the little interval and corresponding string link (i.e., $L_j$ acts on $\sigma^j$).}
Define a new string link 
\[
((L_1,...,L_k) \cdot (\sigma^1,...,\sigma^k)) : I_1 \sqcup ... \sqcup I_n \to \C
\]
by defining its $i$th component $((L_1,...,L_k) \cdot (\sigma^1,...,\sigma^k))_i (t_i) := $
\begin{equation}
\label{eq:sigma-intervals-action}
\left( (L_k \x \mathrm{id}_{D^2}) \circ \overline{\sigma}^k_i \circ (L_k^{-1} \x \mathrm{id}_{D^2}) \right)
\circ ...\circ 
\left( (L_1 \x \mathrm{id}_{D^2}) \circ \overline{\sigma}^1_i \circ (L_1^{-1} \x \mathrm{id}_{D^2}) \right)
(\iota_i(t_i))
\end{equation}
where symbols such as $(L^{-1})^n$ denote maps induced on $n$-fold cartesian products by maps such as $L^{-1}$.

\begin{prop}
Formula \eqref{eq:sigma-intervals-action} defines an operad action of $\LL$ on the space of string links.
\end{prop}
\begin{proof}
First we claim that the formula defines a continuous map from $I_i$ to $D^2 \x I$.  In fact, the disjointness  of the interiors of the $L_j$, together with the conditions (\ref{eq:sigma-endpoints}) and \eqref{eq:sigma-support} on string links guarantee that at every stage of the composition, the input into a $\overline{\sigma}_i$ satisfies condition \eqref{sigma-boundary-are-basepoints}.  It is clear that the image of the map in  \eqref{eq:sigma-intervals-action} lies in $D^2 \x I \subset D^2 \x \R$.  It is also clear that this string link \eqref{eq:sigma-intervals-action} varies continuously as the $L_j$ and $\sigma^j$ vary.

Next we check the three conditions for an operad action, given at the beginning of this subsection.
For the identity condition, consider formula \eqref{eq:sigma-intervals-action} with $k=1$, $L_1=I$ (or $\mathrm{id}_{\R}$, if we view it as a map), and $\sigma=\sigma^1$ any string link.  Then the formula reduces to 
\begin{align*}
(I \cdot \sigma)_i(t) =
\overline{\sigma}_i(\iota(t)) = 
\sigma_i(p \circ \iota(t)) = \sigma_i(t)
\end{align*}
showing that $I \cdot \sigma = \sigma$, as desired.

For the symmetry condition, applying the same permutation in $\Sigma_k$ to the $L_j$ and the $\sigma^j$ will just permute the $k$ terms in the composition.  But the fact that $\overline{\sigma}_i$ is the identity outside of $D^2 \x (0,1)$, together with the disjointness of the interiors of the $L_j$, implies that permuting the terms in the composition leaves the map unchanged.

For the associativity condition, we must check commutativity of the square \eqref{eq:associativity} with $X$ the space of string links.  This is elementary, though complicated to write in detail.  It essentially amounts to the fact that for a string link $\sigma$ and two intervals, say $L_m, L_k$,
\[
\begin{split}
((L_m  \circ L_k) \x \id_{D^2}) \circ \overline{\sigma}_i \circ ((L_m  \circ L_k)^{-1} \x \id_{D^2}) & =\\
 (L_m \x \id_{D^2}) \circ (L_k \x \id_{D^2}) & \circ \overline{\sigma}_i 
\circ (L_k^{-1} \x \id_{D^2})
\circ (L_m^{-1} \x \id_{D^2})
\end{split}
\]
and the fact that $\overline{\sigma}_i$ is the identity outside of $D^2 \x I$.
\end{proof}
It is clear that if we put $k=2$, $L_1=[0,\frac{1}{2}]$, and $L_2=[\frac{1}{2},1]$, we recover the product of stacking string links given in \eqref{eq:s-link-prod}.\\

We now define the little intervals action on our mapping space.
To a map $f=(f_1,..,f_n) \in \map_\mathbf{a}(I^n, \Cnf(\mathcal{C},n))$ we first associate a map 
$\overline{f}: (D^2 \x\R)^n \to \Cnf(D^2 \x \R,n)$.  
Let $\mu: \R \to I$ be the unique continuous monotonic map which is the identity on $I$.
Then define $\overline{f} = (\overline{f}_i,...,\overline{f}_n)$ by 
\[
\overline{f}_i(x_1,...,x_n) :=
\left\{
\begin{array}{ll}
f_i( \mu(p(x_1)),...,\mu(p(x_n))) & \mbox{if\quad $p(x_i) \in I$}, \\
x_i & \mbox{if\quad $p(x_i) \notin I$}.
\end{array}
\right.
\]
Each $\overline{f}_i$ is not continuous at all points $(x_1,...,x_n)$,
but by condition (\ref{eq:endpoints}), it \emph{is} continuous at points satisfying 
\begin{equation}
\label{eq:boundary-points-are-basepoints}
\forall i  \quad
x_i \in D^2 \x \partial I
\Rightarrow x_i \in \{a_i\} \x \partial I.
\end{equation}



Now let $(L_1,...,L_k) \in \LL(k)$ be $k$ little intervals.  
Let $f^1,...,f^k$ be $k$ maps in\\ $\map_\mathbf{a}(I^n, \Cnf(\mathcal{C}, n))$, with 
each $f^j = (f^j_1,...,f^j_n)$.
Define a new map 
\[
((L_1,...,L_k)\cdot (f^1,...,f^k)) : I^n \to \Cnf(\C,n)
\]
by setting
\begin{equation}
\label{eq:intervals-action}
\begin{split}
((L_1,...,L_k)\cdot (f^1,...,f^k))(t_1,...,t_n) & :=\\ 
\bigl( (L_k \x \mathrm{id}_{D^2})^n  \circ \overline{f}^k \circ & (L_k^{-1} \x \mathrm{id}_{D^2})^n \bigr)
\circ \ldots  \\
 \ldots \circ \bigl( (L_1 \x \mathrm{id}_{D^2})^n  & \circ   \overline{f}^1 \circ (L_1^{-1} \x \mathrm{id}_{D^2})^n \bigr)(\iota_1(t_1),\ldots,\iota_n(t_n)).
\end{split}
\end{equation}

\begin{thm}
The formula \eqref{eq:intervals-action} defines an operad action of $\LL$ on $\map_\mathbf{a}(I^n, \Cnf(\mathcal{C},n))$.
\end{thm}
\begin{proof}
We first claim that the formula defines a continuous map from $I^n$ to $\Cnf(\C,n)$.  In fact, by the disjointness of the interiors of the $L_j$ and conditions \eqref{eq:endpoints} and \eqref{eq:support}, the input $(x_1,...,x_n)$ into an $\overline{f}^j$ at any stage in the composition will satisfy condition \eqref{eq:boundary-points-are-basepoints}.
Clearly the image of $(L_1,...,L_k) \cdot (f^1,...,f^k)$ lies in $\Cnf(\C,n) \subset \Cnf(D^2 \x \R,n)$, and 
$(L_1,...,L_k) \cdot (f^1,...,f^k)$ satisfies the conditions \eqref{eq:endpoints} and \eqref{eq:support} because the $f^j$ do.  The fact that this map varies continuously as the $L_j$ and $f^j$ vary is clear from the formula.

We now check the three conditions in the definition of an operad action.
First, for the identity condition, let $k=1$, let $L_1=I$ (or $\mathrm{id}_\R$, if we view it as a map), and let $f=f^1$ be any map in $\map_\mathbf{a}(I^n, \Cnf(\mathcal{C},n))$.  Then formula \eqref{eq:intervals-action} reduces to 
$(I \cdot f) (t_1,..,t_n) = \overline{f} (\iota_1(t_1),...,\iota_n(t_n))$.  Thus for any component $i$
\begin{equation*}
(I \cdot f)_i (t_1,..,t_n) 
= \overline{f}_i (\iota_1(t_1),...,\iota_n(t_n)) 
=\overline{f}_i (p \circ \iota_1 (t_1),..., p \circ \iota_n(t_n)) 
=f_i(t_1,...,t_n)
\end{equation*}
showing that $I \cdot f = f$, as desired.

For the symmetry condition, note that if one acts by the symmetric group $\Sigma_k$ on both the $L_j$ and the $f^j$, the effect on formula \eqref{eq:intervals-action} is to merely permute the $k$ terms in the composition.  But the $L_j$ have disjoint interiors, and the $i$th component of $\overline{f}^j$ acts as the identity on those $x_i$ with $p(x_i) \notin (0,1)$ (where we use condition \eqref{eq:endpoints} for the case $p(x_i)=0$ or 1).  So reordering these $k$ terms has no effect on the composite map.

For associativity, we need to check commutativity of the square \eqref{eq:associativity} with $X=$\\ $\map_\mathbf{a}(I^n, \Cnf(\mathcal{C},n))$.  
But this essentially amounts to the fact that for a map ${f}$ and two little intervals, say $L_m, L_k$, 
\[
((L_m  \circ L_k) \x \id_{D^2}) \circ \overline{f} \circ ((L_m  \circ L_k)^{-1} \x \id_{D^2}) =
(L_m \x \id_{D^2}) \circ (L_k \x \id_{D^2}) \circ \overline{f} 
\circ (L_k^{-1} \x \id_{D^2})
\circ (L_m^{-1} \x \id_{D^2}),
\]
together with the fact that $\overline{f}$ acts by the identity outside of $D^2 \x I$.
\end{proof}

A little intervals operad action on any space $X$ induces a multiplication on $X$ which is associative up to homotopy and has an identity up to homotopy.  In fact, taking $k=2$ and, say, $L_1=[0, \frac{1}{2}] \; (\leftrightarrow l)$ and $L_2=[\frac{1}{2},1] \; (\leftrightarrow u)$ reduces \eqref{eq:operad-action-maps} to a map $X \x X \to X$.  (Associativity and identity are guaranteed by the conditions required for an operad action.)
One can verify that for $k=2, L_1=l, L_2=u$, the expression \eqref{eq:intervals-action} reduces to the product in \eqref{eq:map-prod}.
Thus the product defined in \eqref{eq:map-prod} makes $\map_\mathbf{a}(I^n, \Cnf(\mathcal{C},n))$ into a monoid (with identity), up to homotopy.  

\begin{cor}\label{prop:moniod-str} 
	The product \eqref{eq:map-prod} makes $(\mathcal{M}(n), \cdot)$ into a monoid (with identity).
	\qed
\end{cor}

\begin{rem}
	Our little intervals action (and induced multiplication) on the mapping space is similar to the one defined in \cite[Section 4.2]{Budney-Conant-Koytcheff-Sinha} on the Taylor tower $T_n \mathrm{Emb}(I, I^3)$ for the space of long knots $ \mathrm{Emb}(I, I^3)$.  That paper uses a model $T_n \mathrm{Emb}(I, I^3)$, which is a certain space of maps from an $n$-simplex to a configuration space of $n$ points in $\R^3$.  This model closely resembles the targets of $\kappa$ and $\check{\kappa}$.  All of these little intervals actions are also similar to the one defined by Budney in \cite{Budney:2007} on the space of knots itself, and our formula \eqref{eq:intervals-action} bears a close resemblance to Budney's definition.
\end{rem}


\begin{lem}\label{prop:kappa-homo} 
	The map $\check{\kappa}$ induces a homomorphism of monoids $\mathcal{H}(n)\longrightarrow\mathcal{M}(n)$.
\end{lem}
\begin{proof}
	The proof follows from the comparison of the two formulas for the little intervals actions:
	one for string links given in 
	\eqref{eq:sigma-intervals-action}
	and one for maps given in 
	\eqref{eq:intervals-action}.  
	This comparison shows that $\check{\kappa}$ is compatible with the two little intervals actions.  In particular, it is compatible with the induced product.
\end{proof}

\begin{conj}
	$(\mathcal{M}(n), \cdot)$ is a group, with the inverse $f^{-1}$ to a map $f$ given by $(r \x \id_{D^2}) \circ f^{-1} \circ (r)^n (t_1,...,t_n) = f(1-t_1,...,1-t_n)$, where $r: I \to I$ is the reflection given by $r(t)=1-t$.
\end{conj}

\begin{rem}
	Each stage of the Taylor tower for the space of knots is a group, even though isotopy classes of knots form only a monoid.  This suggests our conjecture above, with our conjectured inverse formula resembling the formula for inverses of braids (and hence homotopy string links).  However, the proof of the analogous fact for knots is somewhat technical \cite[Section 5, esp. Theorem 5.13]{Budney-Conant-Koytcheff-Sinha} and does not immediately extend to the setting of string links.
\end{rem}

\section{Proof of the Main Theorem}\label{sec:proof of main theorem}
\no First, let us define an analog of the projection $\delta_i$ from~\eqref{eq:H(n)-short-exact} for the monoid $\mathcal{M}(n)$. For any $i=1,\ldots,n$, we have the obvious projection 
\begin{align*}
p_i\co (\Cnf(\mathcal{C},n),\mathbf{a})& \longrightarrow (\Cnf(\mathcal{C},n-1),\mathbf{a}_{\widehat{i}})\\
p_i(x_1,\ldots x_n) & =(x_1,\ldots,\widehat{x}_i,\ldots, x_n),
\end{align*}
defined by ``forgetting'' the $i$th coordinate; in particular, the basepoint  $\mathbf{a}=(a_1,\ldots,a_n)$ is mapped to $\mathbf{a}_{\widehat{i}}=(a_1,\ldots,\widehat{a_i},\ldots,a_n)$. Given any $f\in \map_\mathbf{a}(I^n,\Cnf(\mathcal{C},n))$,
we obtain 
$p_i\circ f\co I^n\to \Cnf(\mathcal{C},n-1)$.
We further restrict $p_i\circ f=p_i\circ f(t_1,\dots,t_n)$ to the $i$-th face $I(i)$ of $I^n$, i.e. $I(i)=I_1\times \ldots \times I_{i-1}\times \{0\}\times I_{i+1}\times \ldots\times I_n$.  That is, we precompose by the inclusion $j_i\co I(i)\longmapsto I^n$ to get a map
\[
\widetilde{\psi}_i(f)=p_i\circ f\circ j_i\co I^{n-1}\longrightarrow\Cnf(\mathcal{C},n-1)
\]
satisfying properties \eqref{eq:endpoints}--\eqref{eq:support} with the basepoint $\mathbf{a}_{\widehat{i}}$. The above yields a well defined map on function spaces\footnote{In the language of Taylor towers, this is the map from the ${(1,1,...,1)}$-stage of the $n$-variable Taylor tower for $\slink_\mathbf{a}(\sqcup^n_{i=1} I_i, \C)$ to the ${(1,1,...,1)}$-stage of the $(n-1)$-variable Taylor tower for $\slink_\mathbf{a_{\widehat{i}}}(\sqcup^{n-1}_{i=1} I_i, \C)$ obtained by forgetting the $i$th variable.  We discuss the Taylor tower perspective in Section \ref{sec:Taylor-tower}.} 
\[
\widetilde{\psi}_i\co \map_\mathbf{a}(I^n,\Cnf(\mathcal{C},n))\longrightarrow \map_{\mathbf{a}_{\widehat{i}}}(I^{n-1},\Cnf(\mathcal{C},n-1)).
\]
Let
\begin{equation*}
\psi_i:=\pi_0(\widetilde{\psi}_i)\co \mathcal{M}(n)\longrightarrow \mathcal{M}_i(n-1),
\end{equation*}
where $\mathcal{M}_i(n-1)$ denotes $\pi_0(\map_{\mathbf{a}_{\widehat{i}}}(I^{n-1},\Cnf(\mathcal{C},n-1)))$. It can be shown that $\psi_i$ is a monoid homomorphism, but we have not used this fact in the following argument, and thus we leave the proof to the reader. An analog of the map $\delta$ from \eqref{eq:delta} for $\mathcal{M}(n)$ can be now defined as
\begin{equation*}
\psi:=\psi_1\times\ldots\times \psi_n \co \mathcal{M}(n)  \longrightarrow \prod^{n}_{i=1} \mathcal{M}_i(n-1).
\end{equation*}
\no Consider the following diagram
\begin{equation}\label{diag:check-kappa}
\begin{tikzpicture}[baseline=(current bounding box.center),description/.style={fill=white,inner sep=2pt}]
\matrix (m) [matrix of math nodes, row sep=3em,
column sep=2.5em, text height=1.5ex, text depth=0.25ex]
{ Br\mathcal{H}(n)  & & Br\mathcal{M}(n)\\
	\mathcal{H}(n) & & \mathcal{M}(n)\\ 
	\prod^n_{i=1} \mathcal{H}_i(n-1) & &  \prod^n_{i=1} \mathcal{M}_i(n-1) \\};
\path[->,font=\scriptsize]

(m-1-1) edge node[auto] {$\check{\kappa}\bigl|_{Br\mathcal{H}(n)}$} (m-1-3)
(m-2-1) edge node[auto] {$\check{\kappa}$} (m-2-3)
(m-3-1) edge node[auto] {$(\,\check{\kappa}\,)^n$} (m-3-3)

(m-1-1) edge node[auto] {$\iota$} (m-2-1)
(m-2-1) edge node[auto] {$\delta$} (m-3-1)

(m-1-3) edge node[auto] {$\jmath$} (m-2-3)
(m-2-3) edge node[auto] {$\psi$} (m-3-3);
\end{tikzpicture}
\end{equation} 
where $Br\mathcal{M}(n)$ denotes the image of $Br\mathcal{H}(n)$ in $\mathcal{M}(n)$ under $\check{\kappa}$, and $\iota$ and $\jmath$ are inclusion monomorphisms. The fact that diagram \eqref{diag:check-kappa} commutes follows directly from the definitions of the maps involved; in fact, it is derived as $\pi_0$ of a commuting diagram of maps on the respective function spaces. By Lemma 
\ref{lem:check-kappa-on-BH(n)}, the top map $\check{\kappa}\bigl|_{Br\mathcal{H}(n)}$ in \eqref{diag:check-kappa} is an isomorphism. 

Let us reason inductively  with respect to $n$. All 2-component string links are Borromean, so ${\check{\kappa}=\jmath \circ \check{\kappa}\bigl|_{Br\mathcal{H}(2)}}$ is a monomorphism by Lemma~\ref{lem:check-kappa-on-BH(n)}. 

For $n > 2$, assume that $\check{\kappa}$ is a monomorphism on $\mathcal{H}(n-1)$. Then, after an appropriate choice of basepoints, we can conclude that the bottom map $(\check{\kappa})^n$ is also a monomorphism. Suppose $x, y\in \mathcal{H}(n)$ so that $\check{\kappa}(x)=\check{\kappa}(y)$. Since $\check{\kappa}$ is a monoid homomorphism, we obtain
\[
\check{\kappa}(\mathbf{1})=\check{\kappa}(x\cdot x^{-1})=\check{\kappa}(x)\cdot \check{\kappa}(x^{-1})=\check{\kappa}(y)\cdot \check{\kappa}(x^{-1})=\check{\kappa}(y\cdot x^{-1}),
\]
where $\mathbf{1}$ denotes the identity of $\mathcal{H}(n)$. From diagram \eqref{diag:check-kappa}, 
\[
(\check{\kappa})^n\circ \delta(\mathbf{1})=\psi\circ \check{\kappa}(\mathbf{1})=\psi\circ \check{\kappa}(y\cdot x^{-1})=(\check{\kappa})^n\circ \delta(y\cdot x^{-1}).
\]
Since $(\check{\kappa})^n$ is a monomorphism, we get 
\[
\mathbf{1}=\delta(\mathbf{1})=\delta(y\cdot x^{-1}),
\]
implying $y\cdot x^{-1}\in \ker(\delta)$ and therefore, by \eqref{eq:BH(n)-def}, $y \cdot x^{-1}\in Br\mathcal{H}(n)$. The composition $\jmath\circ \check{\kappa}\bigl|_{Br\mathcal{H}(n)}$ is a monomorphism, thus $y\cdot x^{-1}=\mathbf{1}$ and $x=y$. Hence, $\check{\kappa}$ is a monomorphism on $\mathcal{H}(n)$.~\hfill$\Box$

\section{The $\kappa$-invariant as a map to the Taylor tower}
\label{sec:Taylor-tower}
The Taylor tower is an important object in the study of manifolds from a homotopy-theoretic perspective.  It is a sequence of spaces which approximate spaces of embeddings and other spaces of maps.  Its origins are in the functor calculus of Goodwillie and Weiss.  In our setting, a multivariable version of the Taylor tower is relevant because the source manifold has multiple connected components, see \cite{Weiss:1996} for an introduction to these ideas.  

The target $\map_{\mathbf{a}}(I^n,\Cnf(\mathcal{C},n))$ of our map $\check{\kappa}$ is homotopy equivalent to the $(1,1,...,1)$-stage of the multivariable Taylor tower for the space of link maps of string links.  The purpose of this Section is to establish this homotopy equivalence.

We will use a definition of the Taylor tower for $\slink_\mathbf{a}(\sqcup^n_{i=1} I_i, \C)$ which involves puncturing the source manifold $\sqcup_{i=1}^n I$.  This idea is due to Goodwillie, and our definition will be analogous to the one used by Sinha for the space of knots in \cite{Sinha:2009}.
Let $I_1,...,I_n$ be copies of $I$.
Given natural numbers $m_1,...,m_n$, fix disjoint closed subintervals $A_{1,1},...,A_{1,m_1} \subset I_1, ..., A_{n,1},...,A_{n,m_n} \subset I_n$ which are disjoint from $\partial I_1,...,\partial I_n$ and which, for convenience, are assumed to be in increasing order in each $I_i$.  
Let $\mathcal{P}_\nu[m]$ denote the category (or poset) of nonempty subsets of $[m]=\{0,1,...,m\}$ with inclusions as morphisms.  
Consider the product category $\mathcal{P}_\nu[m_1] \x ... \x \mathcal{P}_\nu[m_n]$, with objects $S=S_1\x...\x S_n$ and inclusions as morphisms.  
For $S_i \in \mathcal{P}[m_i]$, let 
\[
I(S_i) := I_i \setminus \left( \bigcup_{j\in S_i} A_{i,j} \right),
\]
a punctured interval with punctures determined by $S_i$.
Then define $L_S$ as
\[
L_S:=\slink_{\mathbf{a}} (I(S_1) \x ... \x I(S_n), \mathcal{C}),
\]
that is, the space of smooth maps with prescribed behavior on the boundaries, where for $i\neq j$ the images of the punctured $I_i$ and the punctured $I_j$ are disjoint.

For any functor (i.e. diagram) $F$ from a category $C$ to the category of topological spaces, one can define the homotopy limit of $F$ over $C$ as the following space of natural transformations:
\[
\holim_C F := \mathrm{Nat}(|C\downarrow -|, F).
\] 
Here, for any object $c$ of $C$, $C\downarrow c$ is the category of objects over $c$, and $|-|$ denotes geometric realization.  Thus 
\[
\holim_C F \subset \prod_{c \in C} \mathrm{Map}(|C \downarrow c|, F(c)),
\] 
and the elements of the right-hand side which are in $\holim_C F$ are precisely those sequences of maps which are compatible with respect to morphisms of $C$.
Note that when $C=\P_\nu[m_1] \x ... \x \P_\nu[m_n]$, the realization $|C|$ is a product of $m_i$-dimensional simplices.
Moreover, in this case, $C$ has a final object $[m_1]\x...\x [m_n]$, so $|C \downarrow c| \subset |C|$.

\begin{defin}
	\label{TaylorTowerDef}
	Define the $(m_1,...,m_n)$ stage of the (multivariable) Taylor tower for\\ $\slink_\mathbf{a}(\sqcup^n_{i=1} I_i, \C)$ as
	\[
	T_{(m_1,..,m_n)} \slink_\mathbf{a}(\sqcup^n_{i=1} I_i, \C):=
	\displaystyle \holim\limits_{S \in \mathcal{P}[m_1] \x... \x\mathcal{P}[m_n]} L_S.
	\]
\end{defin}

The direct extension of Weiss's original definition of the Taylor tower \cite{Weiss:1999} to this multivariable setting is a ``larger'' space because it is a homotopy limit over a much larger category of open subsets of the source manifold.  An equivalence between Definition \ref{TaylorTowerDef} and Weiss's definition can be readily established \cite[Proposition 7.1]{Munson-Volic:2012}.  Therefore it suffices to connect our mapping space to the space in Definition \ref{TaylorTowerDef}.

\begin{prop}
\label{prop:taylor}
	$\map_{\mathbf{a}}(I^n,\Cnf(\mathcal{C},n)) \simeq 
	T_{(1,1,...,1)}  \slink_\mathbf{a}(\sqcup^n_{i=1} I_i, \C)$.
\end{prop}
\begin{proof}
	This equivalence is essentially a matter of replacing the spaces $L_S$ in the definition of the Taylor tower by certain deformation retracts.  For any $S=S_1\x...\x S_n$, consider the homotopy which on each $I(S_i)$ contracts the two endpoint components to the respective endpoints and contracts the remaining components to their midpoints.  By precomposing the link maps in $L_S$ by this homotopy, we obtain for all $S$ a deformation retraction $r_S: L_S \to \overline{L_S}$ onto the subspace of maps which are constant on each component, which we may write as   
	\[
	\overline{L_S} :=
	\slink_\mathbf{a}(\pi_0(I(S_1)) \sqcup ... \sqcup \pi_0(I(S_n)), \C).
	\]
	The maps in $\overline{L_S}$ are fixed on each of the $2n$ endpoint components in $\sqcup_{i=1}^n I(S_i)$, and they also satisfy the disjointness condition that images of points in $\pi_0(S_i)$ are disjoint from images of points in $\pi_0(S_j)$ for all $i \neq j$.
	
	If $T=S\cup \{k\}$, we associate to the inclusion $S \subset T$ the map $\overline{L_S} \to \overline{L_T}$ induced by doubling the appropriate point in the domain, so as to make the $r_S$ fit together into a natural transformation.
	For example, for $n=1$ and the inclusion $\{0\} = S \to T=\{0,1\}$, the map $L_S \to L_T$ is induced by the restriction from $I \setminus A_0$ to $I \setminus (A_0 \cup A_1)$.  Thus we take the map $\ast \cong \overline{L_S} \to \overline{L_T} \cong \mathrm{Map}(\ast, \mathcal{C})$ to be the map induced by doubling the second of the two points in the domain.  In other words, its image is the map which sends the point to the basepoint $(a,1)$.  On the other hand, for $S=\{1\}, T=\{0,1\}$, the image of the map $\overline{L_S} \to \overline{L_T}$ is the map which sends the point to the basepoint $(a,0)$.  Though $n$ may be arbitrary, we only need to consider case where each $S_i \subset \{0,1\}$, so we refrain from further details about these doubling maps in general.
	
	Now the maps $r_S$ on each object $S$, together with the maps $\overline{L_S} \to \overline{L_T}$ for each morphism $S \subset T$, define a natural transformation of diagrams $\{L_S\} \to \{\overline{L_S}\}$.  Since each $r_S$ is a homotopy equivalence,
$T_{(m_1,..,m_n)} \slink_\mathbf{a}(\sqcup^n_{i=1} I_i, \C) \simeq \holim_{S \in \mathcal{P}[m_1] \x... \x\mathcal{P}[m_n]} \overline{L_S}$.
	
	We now specialize to the case $m_1=...=m_n=1$.  We claim that in this case the latter homotopy limit is (nearly) homeomorphic to  $\map_{\mathbf{a}}(I^n,\Cnf(\mathcal{C},n))$.  First note that
$\overline{L_{\{0,1\}^n}}$ is precisely $\Cnf(\mathcal{C},n)$, and that $| (\P_\nu[1])^n |$ is the $n$-cube.  Thus part of the data of an element of $\holim_S \overline{L_S}$ is a map from the $n$-cube to $\Cnf(\mathcal{C},n)$.  Now for any $T\subset S$ the map $\overline{L_T} \to \overline{L_S}$ is an inclusion, and $\{0,1\}^n$ is the final object.  Thus an element of $\holim_S \overline{L_S}$ is completely specified by a map $I^n \to \Cnf(\mathcal{C},n)$, and the role of $ T \subsetneq \{0,1\}^n$ is to specify which such maps are elements of the homotopy limit.  But unwinding the definitions shows that the conditions specified by the various subsets $T$ are precisely the basepoint conditions (\ref{eq:endpoints}).
The support condition (\ref{eq:support}) is not automatically satisfied by elements of $\holim_S \overline{L_S}$, but one can construct a deformation retraction onto a subspace of elements which do satisfy it.
\end{proof}

We may now combine Proposition \ref{prop:taylor} with our main theorem:
\begin{cor}
The map from the space of link maps of string links to the $(1,1,...,1)$-stage of its Taylor tower is injective on path-components.
\end{cor}

\appendix
\section{About the basepoints}\label{app:basepoints}

In Section \ref{sec:kappa} we considered pointed homotopy classes in the domain and codomain of the $\kappa$-invariant \eqref{eq:kappa}. Here, we intend to show that the basepoint-free version of $\kappa$ is equivalent to the pointed one.

Let $s_i \in S^1$ for each $i=1,\ldots , n$ and let $\link(\sqcup^{n}_{i=1} S^1; \R^3)$ denote the function space of link maps $L: (S^1,s_1)\sqcup\ldots \sqcup (S^1,s_n)  \longrightarrow \R^3$ (defined in \eqref{eq:link-map}) without the basepoint condition $L_i(s_i)=a_i$. With $\mathbf{s}=(s_i)_i$ fixed we have a well defined evaluation map
\[
ev_{\mathbf{s}}\co \link(\sqcup^{n}_{i=1} S^1; \R^3)\longmapsto \Cnf(n).
\]
By standard results---see e.g.  \cite{Arkowitz:2011}---$ev_{\mathbf{s}}$ is a fibration, with fiber $\link_{\mathbf{a}}(\sqcup^{n}_{i=1} S^1; \R^3)$. The map induced from fiber inclusion in the long exact sequence of the fibration 
\[
j_\ast:\pi_0(\link_{\mathbf{a}}(\sqcup^{n}_{i=1} S^1; \R^3))\longrightarrow \pi_0(\link(\sqcup^{n}_{i=1} S^1; \R^3)),
\]
can be shown to be surjective.\footnote{One may easily change the basepoint of a link after an isotopy.}  The map $j_\ast$ is also injective since $j_\ast(L)=j_\ast(L')$ if and only if $L$ and $L'$ differ by an element in $\pi_1(\Cnf(n))$ (c.f. \cite[p. 132]{Arkowitz:2011}), which is the trivial group. A completely analogous argument applies to the spaces $\map_\mathbf{a}(\T^n, \Cnf(n))$ and $\map(\T^n, \Cnf(n))$, and therefore we have the following bijections
\[ 
\begin{split}
\pi_0(\link(\sqcup^{n}_{i=1} S^1; \R^3)) & \cong \pi_0(\link_{\mathbf{a}}(\sqcup^{n}_{i=1} S^1; \R^3)),\\
\pi_0(\map(\T^n, \Cnf(n))) & \cong \pi_0(\map_\mathbf{a}(\T^n, \Cnf(n))).
\end{split}
\]
The above maps fit into the diagram
\begin{equation*}
\begin{tikzpicture}[baseline=(current bounding box.center),description/.style={fill=white,inner sep=2pt}]
\matrix (m) 
[
matrix of math nodes
, row sep=3em
, column sep=2.5em, 
text height=1.0ex, text depth=0.25ex
]
{ \pi_0(\link_{\mathbf{a}}(\sqcup^{n}_{i=1} S^1; \R^3))  & & \pi_0(\map_\mathbf{a}(\T^n, \Cnf(n))) \\
	\pi_0(\link(\sqcup^{n}_{i=1} S^1; \R^3)) & & \pi_0(\map(\T^n, \Cnf(n))) \\ 
};
\path[->,font=\scriptsize]

(m-1-1) edge node[auto] {$\kappa$} (m-1-3)
(m-2-1) edge node[auto] {$\kappa_{\text{free}}$} (m-2-3)

(m-1-1) edge node[auto] {$\cong$} (m-2-1)

(m-1-3) edge node[auto] {$\cong$} (m-2-3);
\end{tikzpicture}
\end{equation*}
and hence the pointed $\kappa$-invariant is injective if and only if the free $\kappa$-invariant is injective.
\bibliography{s-links}
\bibliographystyle{plain}
 
\end{document}